\documentclass[reqno,twoside]{amsart}
\usepackage{amsfonts}
\usepackage{amsmath,amssymb}
\usepackage{epic}
\usepackage{pstricks}
\usepackage{graphicx}
\usepackage{xcolor}
\usepackage{a4wide,amsmath,amssymb,latexsym,amsthm}
\usepackage{eucal}
\usepackage{graphicx}

\newtheorem{theorem}{Theorem}[section]
\newtheorem{proposition}[theorem]{Proposition}
\newtheorem{remark}[theorem]{Remark}
\newtheorem{lemma}[theorem]{Lemma}
 
\newtheorem{definition}[theorem]{Definition}

\numberwithin{equation}{section}

\graphicspath{{./Imgs/}}

\newcommand{\R}{\mathbb R}

\newcommand{\be}{\begin{equation}}
\newcommand{\ee}{\end{equation}}
\newcommand{\ba}{\begin{eqnarray}}
\newcommand{\ea}{\end{eqnarray}}
\newcommand{\beq}{\begin{equation}}
\newcommand{\eeq}{\end{equation}}

\usepackage{color}
\definecolor{red}{rgb}{0,0,0}

\usepackage[textsize=small]{todonotes}

\numberwithin{equation}{section}

\usepackage{color}
\usepackage{stmaryrd}


\begin{document}

\title[Inverse problem in water-waves]{The stability for an inverse problem of bottom recovering  in water-waves}

\author{R. Lecaros}
\address{R. Lecaros, Universidad T\'ecnica Federico Santa Mar\'ia, Departamento de Matem\'atica, Casilla 110-V, Valpara\'iso, Chile}
\email{rodrigo.lecaros@usm.cl}

\author{J. L\'opez-R\'ios}
\address{J. L\'opez-R\'ios, Escuela de Ciencias Matem\'aticas y Computacionales, YACHAY TECH, San Miguel de Urcuqu\'i, Hacienda San Jos\'e S/N, Ecuador}
\email{jlopez@yachaytech.edu.ec}

\author{J.H. Ortega}
\address{J.H. Ortega, Departamento de Ingenier\'ia Matem\'atica and Centro de Modelamiento Matem\'atico, Universidad de Chile and UMI-CNRS 2807, Beauchef 851, Ed. Norte, 5th floor, Santiago, Chile}
\email{jortega@dim.uchile.cl}

\author{S. Zamorano}
\address{S. Zamorano, Universidad de Santiago de Chile, Departamento de
Matem\'atica y Ciencia de la Computaci\'on, Facultad de Ciencia, Casilla 307-Correo 2,
Santiago, Chile.}
 \email{sebastian.zamorano@usach.cl}

\thanks{This work was partially supported by the project MATHAMSUD MADESEIS 19-MATH-01 and CMM-Basal Grant PIA AFB-170001. R. Lecaros was partially supported by FONDECYT Grant 11180874, J. H. Ortega was partially supported by Fondecyt Grant 1201125, S. Zamorano was partially supported Conicyt PAI Convocatoria Nacional Subvención a la Instalación en la Academia Convocatoria 2019 PAI 77190106.}

\begin{abstract}
In this article we deal with a class of geometric inverse  problem for bottom detection by one single measurement on the free surface in water--waves. 
We found upper and lower bounds for the size of the region enclosed between two different bottoms, in terms of Neumann and/or Dirichlet data on the free surface.
	Starting from the general water--waves system in bounded domains with side walls, we manage to formulate the problem in terms of the Dirichlet to Neumann operator and thus, as an elliptic problem in a bounded domain with Neumann homogeneous condition on the rigid boundary. Then we study the properties of the Dirichlet to Neumann map and analyze the called method of size estimation.
\end{abstract}

\maketitle

%%%%%%%%%%%%%%%%%%%%
\section{Introduction}

\subsection{Motivation}
The study of waves in the ocean is a wide open area of interest, not only for its practical applications, but also for the theoretical development it represents in oceanography and the mathematical study of Partial Differential Equations (PDE's). Among other applications of this knowledge, we can find the modeling of tsunamis, the influence of the ocean floor and plate displacements on the occurrence of certain waves, and the modeling of wave--breaking phenomena near the coast.

While there is no a unified approach to deal with this problem, from the mathematical point of view there are some general considerations that are widely accepted (\cite{lannes2013water}). Let us consider that we have an ideal, inviscid, incompressible fluid. Then, the general water--waves problem is the description of the motion of a layer of fluid, delimited below by a solid bottom, and above by a free surface; influenced by the force of gravity. It is modeled by means of conservation laws, together with suitable boundary conditions.
	
More specifically, if we consider that the bottom and the wave surface are parameterized, respectively, as $b(x)$, $\zeta(t,x)$, and we define $\Omega_t=\{(x,y)\in \R^N\times\R :b(x)<y<\zeta(t,x)\}$, then the general water--waves system is given by (see \cite{lannes2013water})
\begin{equation}
\label{1}
\left\{
\begin{array}{rll}
\Delta_{x,y}\phi=0, && \Omega_t, \\
\partial_t\zeta+\nabla_x\zeta\cdot\nabla_x\phi =\partial_y\phi, && y=\zeta, \\
\partial_t\phi+\frac{1}{2}\left(|\nabla_x\phi|^2+(\partial_y\phi)^2\right)+g\zeta  =0, && y=\zeta, \\
\partial_n\phi  =0, && y=b,
\end{array}
\right.
\end{equation}
where $g$ is the gravity constant, and the velocity of the fluid, ${\bf u}$, is such that ${\bf u}=\nabla_{x,y}\phi$, with $\phi$ being the velocity potential. We assume $b(x)<\zeta(t,x)$ for any $t>0$ and $x\in\R^N$, see Figure 1.

\begin{figure}
	\centering
	\includegraphics[scale=1.0]{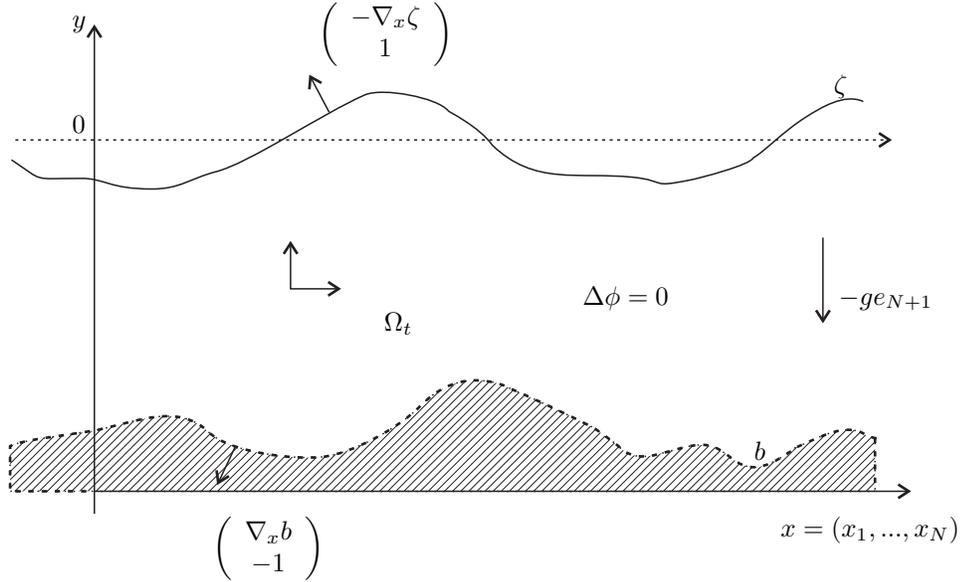}
	\put(-318,145){$0$}
	\put(-30,160){$\zeta$}
	\put(-60,21){$b$}
	\put(-318,185){$y$}
	\put(-50,-8){$x=(x_1,...,x_N)$}
	\put(-225,180){$\left( \begin{array}{c} -\nabla_x\zeta \\ 1 \end{array}\right)$}
	\put(-265,-15){$\left( \begin{array}{c} \nabla_xb \\ -1  \end{array}\right)$}
	\put(-28,80){$-ge_{N+1}$}
	\put(-200,70){$\Omega_t$}
	\put(-125,80){$\Delta\phi=0$}
		\label{Fig0}
		\caption{Scheme of the water--waves model}
\end{figure}

In this work the inverse problem of bottom size estimation from measurements on the free surface (the waves) is addressed. Namely, let $S\subset\R^N$ be an open, bounded set. Assume $x\in S$ and there are two bottoms $b_0,b_1$ such that $b_0(x)\le b_1(x)$ for all $x\in S$. Let $(\phi_0,\zeta_0),(\phi_1,\zeta_1)$ be the two solutions of (\ref{1}) associated to $b_0$ and $b_1$, respectively. If the free--surface for both problems coincide in some $t_0>0$, that is $\zeta_0(t_0,\cdot)=\zeta_1(t_0,\cdot)$, we want to find two positive constants $C_1,C_2$ such that the volume of region  $D=\{(x,y)\in S\times\R:b_0(x)\le y\le b_1(x)\}$ can be estimated, from above and below, by measurements of $\phi_0(t_0,\cdot),\phi_1(t_0,\cdot)$ on the common free--surface. Specifically, let us consider the following numbers
\begin{equation*}
W_{i,j}=\int_{y=\zeta(t_0,\cdot)}\phi_i(t_0,\cdot)\partial_n\phi_j(t_0,\cdot), \quad i,j=0,1.
\end{equation*}

Our main results consist in estimating the volume of $D$ in terms of these numbers, that is,
\begin{equation}\label{IP2}
C_1\eta_1(W_{1,1}-W_{1,0}, W_{1,1}-W_{0,1})\leq |D|\leq C_2\eta_2(W_{1,1}-W_{1,0}, W_{1,1}-W_{0,1}),
\end{equation}
for some suitable functions $\eta_1,\eta_2$ such that $\eta_i(0,0)=0$.

%\begin{eqnarray}\label{IP}
%C_1 \eta_1\left(\int_{y=\zeta(t_0,\cdot)}\phi(t_0,\cdot)(\partial_n \phi(t_0,\cdot)-\partial_n\phi_0(t_0,\cdot)), \int_{y=\zeta(t_0,\cdot)}\partial_n\phi(t_0,\cdot)(\phi(t_0,\cdot)-\phi_0(t_0,\cdot))\right)\leq |D|\nonumber\\\leq C_2\eta_2\left(\int_{y=\zeta(t_0,\cdot)}\phi(t_0,\cdot)(\partial_n \phi(t_0,\cdot)-\partial_n\phi_0(t_0,\cdot)), \int_{y=\zeta(t_0,\cdot)}\partial_n\phi(t_0,\cdot)(\phi(t_0,\cdot)-\phi_0(t_0,\cdot))\right),
%\end{eqnarray}
%where the functions $\eta_1,\eta_2$ such that $\eta_i(0,0)=0$.

%The identifiability of this inverse problem has been addressed in \cite{fontelos2017bottom}, in the case of an infinite strip and by measuring simultaneously the profile and the velocity potential on an open interval of time. However, for the bounded domain case, and for elliptic equations inside the domain, there are examples that show that the stability is no better than logarithmic (see \cite{alessandrini2002detecting}, \cite{beretta1998stable}), which limits the possibility of finding efficient reconstruction methods.

Notice that measurements are made in a single time $t_0$. This can be explained by the relation between the velocity potential inside the domain and its trace on the free surface as we do next. We rewrite system (\ref{1}) in a different way and the first step is to consider the two nonlinear boundary conditions on the free surface, $y=\zeta$, as an independent evolutionary system, which is related to the inner domain, $\Omega_t$, through an elliptic problem. This is possible by considering the called Dirichlet to Neumann operator, see \cite{craig1993numerical,zakharov1968stability}. That is, if we assume that $\psi(t,x)=\phi(t,x,\zeta(t,x))$ is known, then the elliptic problem 
\begin{equation}
\label{4}
\left\{
\begin{array}{rll}
\Delta_{x,y}\phi=0, &&\Omega_t, \\
\phi=\psi,  &&y=\zeta, \\
\partial_n\phi=0, && y=b,
\end{array}
\right.
\end{equation}
has a unique solution. Therefore, the following Dirichlet to Neumann operator, $G$, is well defined on $(0,\infty)\times\R^N$. Let
\begin{equation}
\label{3}
G(\zeta,b)\psi:=-\sqrt{1+|\nabla_x\zeta|^2}\partial_n\phi|_{y=\zeta}.
\end{equation}

Moreover, by the chain rule
\begin{equation*}
\nabla_x\psi=\nabla_x\phi+\phi_y\nabla_x\zeta,
\end{equation*}
which implies
\begin{equation*}
\nabla_x\psi\cdot\nabla_x\zeta+G=\nabla_x\phi\cdot\nabla_x\zeta+\phi_y|\nabla_x\zeta|^2+G=\phi_y(1+|\nabla_x\zeta|^2).
\end{equation*}

This last equation allows to relating the vertical velocity at the free boundary, and the operator $G$ by
\begin{equation}
\label{3_1}
\phi_y=\frac{G+\nabla_x\psi\cdot\nabla_x\zeta}{1+|\nabla_x\zeta|^2}.
\end{equation}

Finally, from (\ref{3}) and (\ref{3_1}), the two boundary conditions on the free surface, in system (\ref{1}), are written as
\begin{equation}
\label{2}
\left\{
\begin{array}{r}
\partial_t\zeta-G(\zeta,b)\psi=0, \\
\displaystyle \partial_t\psi+g\zeta+\frac{1}{2}|\nabla_x\psi|^2-\frac{(G(\zeta,b)\psi+\nabla_x\zeta\cdot\nabla_x\psi)^2}{2(1+|\nabla_x\zeta|^2)}=0,
\end{array}
\right.
\end{equation}
where $(t,x)\in(0,\infty)\times\R^N$.

The above system is complemented with initial conditions $\zeta(0,x)=\zeta_0(x)$, $\psi(0,x)=\psi_0(x)$. It is worth noting that problem (\ref{2}) is an evolutionary system, involving a differential, nonlinear, nonlocal operator. Moreover, the information of the bottom is implicit as a parameter through the Dirichlet to Neumann operator $G$.

In literature (see \cite{lannes2013water}), (\ref{2}) it is commonly known as the general water--waves system, in replacement of (\ref{1}). This is so because, once (\ref{2}) is solved, it is possible to find $\phi$ by means of system (\ref{4}). The fact that we only need a single measurement on the free--surface at time $t_0$ is clear now. That is, whether we know $\partial_t\zeta|_{t=t_0}$ or $\zeta$ on an open interval of time, by the first equation in (\ref{2}), the Neumann condition on the free boundary is known, which allows us to establish the inverse problem of recovering $b$, as a boundary detection by measures of $\psi,\partial_n\phi$ on $y=\zeta$, for the system (\ref{4}).

The assumption of the existence of $t_0>0$ and $S\subset\R^N$ to obtain the identifiability of the inverse problem of bottom detection, is familiar in the context of water--waves. For instance, in \cite{vasan2013inverse}, the authors addressed the numerical problem of recovering the bottom from the water wave height and its first two time derivatives
at one time instant, assuming that the velocity potential is periodic and the profile height is small. They also consider the cases of measuring the surface profile over an interval $[0,T]$ or the surface profile at a discrete set of points. In all cases they
attempt to eliminate the necessity of measuring the velocity potential $\psi$, which, as they explain, is physically impractical. Another example is given in \cite{kenig2020uniqueness}, where the authors proved the unique continuation property for the Benjamin--Ono equation for a nonlocal operator, under the assumption that the solution is zero on an open subset of $[0,T]\times\R$.

Concerning the inverse problem of bottom detection through measures on the free surface, in \cite{fontelos2017bottom}, the authors used the simple elliptic formulation (\ref{3})--(\ref{4}) of the water--waves system and a classical strategy in geometric inverse problems to prove the identifiability of the bottom by measuring, simultaneously, the profile and its time derivative at the free surface in a single time, and an open subset of $\R^N$. That is, for a fixed $t_0>0$ and $x\in S\subset\R^N$ an open set, they proved the injectivity of the operator $b\mapsto (\zeta,\psi, \partial_t\zeta)_{t=t_0}$; namely, if $\zeta_1(t_0,x)=\zeta_2(t_0,x)$, $\psi_1(t_0,x)=\psi_2(t_0,x)$, $\partial_t\zeta_1(t_0,x)=\partial_t\zeta_2(t_0,x)$, then one has $b_1(x)=b_2(x)$ in $\R^N$.

%In practice, this result implies the identification of the bottom by measuring, simultaneously, Dirichlet and Neumann data on an open set of the free boundary. Note also that the measurements have to be done on a common region of the free surface.

Existence and uniqueness of solutions for system (\ref{2}), within a Sobolev class, have been widely studied. We refer, for instance, to the literature review by Lannes \cite{lannes2013water} in the N--dimensional case. For the well posedness of system (\ref{2}) in one dimension, on bounded domains, see Alazard {\it et. al.} \cite{alazard2016cauchy}. Well posedness of the N--dimensional case on bounded domains is still an open problem, among other things, due to the presence of solid walls and the underlying physics in the interaction between the wall and the free surface \cite{benjamin1979gravity,graham1983new,kim2015capillary}.

%{\color{red}Yo pensaría que por aquí se puede ir cortando la introducción. Tal vez tomar lo que sirva de mas abajo para cerrarla y luego agregar una sección inicial con lo que queda de aquí para abajo.
%Notar que todavía no hemos respondido bien lo de la no dependencia del tiempo.
%}

\subsection{Size estimate for the elliptic formulation}

As we mentioned before, we are interested in the inverse problem of estimating the volume of the region $D$ from the single measurement of $\Lambda_{t_0}(b):= (\zeta,\phi|_{\zeta},\partial_n\phi|_{\zeta})|_{t=t_0}$. The study of this inverse geometric problem is motivated, firstly, by the well--know result that stability of the elliptic problem (\ref{4}) is a severe ill--posed problem. That is, given $f,h,r$, the continuity of the solution of
\begin{equation}\label{Cauchy}
\left\{
\begin{array}{rll}
\Delta u=f,  &&\Omega, \\
u=h,  && \partial\Omega, \\
\partial_n u=r,  &&\partial\Omega,
\end{array}
\right.
\end{equation}
in terms of the data $(f,h,r)$ is not, in general, of a Lipschitz type. That is, the Cauchy problem (\ref{Cauchy}) is ill--posed in the Hadamard sense. Besides, as Hadamard pointed out in \cite{hadamard1923lectures}, confirmed later in \cite{beretta2017size,bourgeois2010stability, choulli2016applications, choulli2019new, choulli2019global}, the modulus of continuity of the mapping $(f,h,r)\mapsto u$ is of a logarithmic type and is the best possible one to be expected.

%
%{\color{blue}Therefore, we are going to state the stability we want to deal with in this work. Namely, the relation between problems (\ref{2}) and (\ref{4}), suggest that, once $\zeta,\psi$ are known in $(0,T)$ after solving the evolutionary system, it is possible to use this information to set a well-posed problem (\ref{4}). Moreover the first equation in (\ref{2}), allows us to know $\partial_n\phi|_{y=\zeta}$. Finally, since the bottom is time independent, to achieve the stability as was suggested in \cite{fontelos2017bottom}, it is enough to set the problem on a single time $t=t_0$.} Thus, we want to analyze the stability of the operator $ \Lambda_{t_0}(b):=(\zeta,\psi,\partial_t\zeta)|_{t=t_0}$ to bound the variation of two different bottoms .

Secondly, the three variables  $(\zeta,\phi|_{\zeta},\partial_n\phi|_{\zeta})$ are not independent at all. The complete stability of the functional $\Lambda_{t_0}(b)$ is a difficult problem to deal with; not only because the data are different, but, their free surfaces do not necessarily intersect each other. For example, if we consider $b_0,b_1:S\subset \R^N\to \R$ two different  bottoms, we have different measures $\Lambda_{t_0}(b_0)\neq \Lambda_{t_0}(b_1)$, that is 
\begin{equation*}
(\zeta_0,\phi_0|_{\zeta},\partial_n\phi_0|_{\zeta})|_{t_0}\neq (\zeta_1,\phi_1|_{\zeta},\partial_n\phi_1|_{\zeta})|_{t_0},
\end{equation*}
where these functions satisfy the corresponding systems:
\begin{equation}
\label{6}
\left\{
\begin{array}{rll}
\Delta\phi_i=0,  &&\Omega_{t_0}^i, \\
\phi_i=\psi_i, &&y=\zeta_i, \\
\partial_n\phi_i=0, &&y=b_i,
%	&
%	&
%	\left\{\begin{array}{cc}
%		\Delta\phi_1=0, \quad &\Omega(b_1,\zeta_1), \\
%		\phi_1=\psi_1, \quad &\Gamma(\zeta_1), \\
%		\partial_n\phi_1=0, \quad &\Gamma(b_1)\cup\Gamma_w(b_1,\zeta_1),\\
%		G(\zeta_1,b_1)\psi_1=&-\sqrt{1+|\nabla_x\zeta_1|^2}\partial_n\phi_1|_{y=\zeta_1}
%	\end{array}\right.
%	\end{array}
\end{array}
\right.
\end{equation}
for $i=0,1$, at $t=t_0$.

Therefore, in relation with the bottom, it seems reasonable the study of a different quantity. Namely, the volume enclosed by two bottoms. Following the approach introduced by Alessandrini \emph{et. al.} in \cite{alessandrini2002detecting}, we establish a quantitative estimate of the size of the region bounded by the side walls and the bottoms in terms of suitable measurements. 

The following three cases enclose the main idea of this work. We are going to work in bounded domains, so we assume that $S\subset\R^N$ is a given open set. To ease notation and stress the dependence of the operator in terms of the quantities on the upper boundary, we define
\begin{equation*}
\Omega(b,\zeta)=\{(x,y)\in S\times\R :b(x)<y<\zeta(t_0,x)\},
\end{equation*}
\begin{equation*}
\Gamma_{w}(b,\zeta):=\{(x,y)\in S\times\R:b(x)<y<\zeta(t_0,x),\; x\in\partial S\},
\end{equation*}
\begin{equation*}
\Gamma(\zeta)=\{(x,y)\in S\times\R: y=\zeta(t_0,x),\; x\in S\}.
\end{equation*}

\smallskip 

{\bf Case I: single bottom cavity. Dirichlet and/or Neumann measurements on the whole free surface.} $\zeta_0=\zeta_1$ and $b_0\geq b_1$ or $b_0\leq b_1$.

Let $\Omega(b_0,\zeta_0)$ be the domain such that $\partial\Omega(b_0,\zeta_0)=\Gamma(\zeta_0)\cup\Gamma_w(b_0,\zeta_0)\cup\Gamma(b_0)$ and  $\Omega(b_1,\zeta_0)$ the domain with boundary $\partial\Omega(b_1,\zeta_0)=\Gamma(\zeta_0)\cup\Gamma_w(b_1,\zeta_0)\cup\Gamma(b_1)$. Let us consider $b_1\geq b_0$. We consider $\phi_0,\phi$ be the unique weak solutions of the following problems
\begin{equation}\label{caseI}
\!\!\!\!\!\!\!\!\!\!\!\!\! \left\{
\begin{array}{rll}
\Delta\phi_0=0, &&\Omega(b_0,\zeta_0), \\
\phi_0=\psi_0, &&\Gamma(\zeta_0), \\
\partial_n\phi_0=0, &&\Gamma(b_0)\cup\Gamma_w(b_0,\zeta_0),
\end{array}
\right.
 \qquad
\left\{
\begin{array}{rll}
\Delta\phi=0, &&\Omega(b_1,\zeta_0), \\
\phi=\psi, &&\Gamma(\zeta_0), \\
\partial_n\phi=0, &&\Gamma(b_1)\cup\Gamma_w(b_1,\zeta_0).
\end{array}
\right.
\end{equation}

We will study the lower and upper estimate for the volume of $D:=\Omega(b_0,\zeta_0)\setminus \overline{\Omega(b_1,\zeta_0)}$, when the Dirichlet and Neumann measurements $\psi_0,\psi_1$ and $\partial_n\phi_0,\partial_n\phi$, respectively, are performed in the same surface $\Gamma(\zeta_0)$. Specifically, we want to obtain two positive constants $C_1,C_2>0$  such that
\begin{eqnarray*}
C_1 \eta_1\left(\int_{\Gamma(\zeta_0)}\psi(\partial_n\phi-\partial_n\phi_0), \int_{\Gamma(\zeta_0)}\partial_n\phi(\psi-\psi_0)\right)\leq |D|\\\qquad\leq C_2 \eta_2\left(\int_{\Gamma(\zeta_0)}\psi(\partial_n\phi-\partial_n\phi_0), \int_{\Gamma(\zeta_0)}\partial_n\phi(\psi-\psi_0)\right),
\end{eqnarray*}
for some functions $\eta_1,\eta_2$ such that $\eta_i(0,0)=0$.\\

{\bf Case II: multiple cavities. Dirichlet and/or Neumann measurements on the whole free surface.} $\zeta_0=\zeta_1$ and the set $B=\{x\in S:\; b_0(x)=b_1(x)\}$ is finite.

In this case, we consider $\phi_0,\phi$ the solutions of problems (\ref{caseI}) as in Case I. The volume estimate of $D$ will be studied when the Dirichlet and Neumann measurements are made on the same free surface $\Gamma(\zeta_0)$, and we allow bottoms intersections. That is, the subset $D$ is given by  $D:=\Omega(b_0,\zeta_0)\vartriangle \Omega(b_1,\zeta_0)$. Then, in this case, we want to prove the existence of constants $C_3,C_4>0$, such that 
\begin{eqnarray*}
C_3 \eta_3\left(\int_{\Gamma(\zeta_0)}\psi(\partial_n\phi-\partial_n\phi_0), \int_{\Gamma(\zeta_0)}\partial_n\phi(\psi-\psi_0)\right)\leq |\Omega(b_0,\zeta_0)\vartriangle \Omega(b_1,\zeta_0)|\\\qquad\leq C_4 \eta_4\left(\int_{\Gamma(\zeta_0)}\psi(\partial_n\phi-\partial_n\phi_0), \int_{\Gamma(\zeta_0)}\partial_n\phi(\psi-\psi_0)\right),
\end{eqnarray*}
for some functions $\eta_3,\eta_4$ such that $\eta_i(0,0)=0$.

\smallskip\

{\bf Case III: partial measurements on the free surface.} 

In this final case, we analyze the previous two cases when the measurements are performed in a subset $\Gamma^{*}$ of the free surface $\Gamma(\zeta_0)$. That is, we consider $\phi_0,\phi$ being the solutions of:
\begin{equation}\label{caseI.1}
\!\!\!\!\!\!\!\!\!\!\!\!\! \left\{
\begin{array}{rll}
\Delta\phi_0=0, &&\Omega(b_0,\zeta_0), \\
\phi_0=\psi_0, &&\Gamma^{*}, \\
\partial_n\phi_0=0,  &&\Gamma(b_0)\cup\Gamma_w(b_0,\zeta_0),
\end{array}
\right.
\qquad
\left\{
\begin{array}{rll}
\Delta\phi=0, &&\Omega(b_1,\zeta_0), \\
\phi=\psi, &&\Gamma^{*}, \\
\partial_n\phi=0, &&\Gamma(b_1)\cup\Gamma_w(b_1,\zeta_0).
\end{array}
\right.
\end{equation}

We want to establish again an estimate for the size of $D$ in terms of the Dirichlet and/or Neumann data on $\Gamma^{*}$. However, we were able to derive only the upper estimate:
\begin{equation*}
 |D|\leq C_5 \eta_5\left(\int_{\Gamma^{*}}\psi(\partial_n\phi-\partial_n\phi_0), \int_{\Gamma(\zeta_0)}\partial_n\phi(\psi-\psi_0)\right),
\end{equation*}
for some function $\eta_5$ such that $\eta_5(0,0)=0$.

Note, from the cases above, that functions $\eta_i$ satisfy $\eta_i(0,0)=0$. In particular, if the Neumann and Dirichlet measurements are equal on the free surface $\Gamma(\zeta_0)$, then we have $|D|=0$. Therefore, we obtain a unique continuation property in the sense that if the measurements are equal, the bottoms are the same.

As was mentioned above, we will use some techniques developed in \cite{alessandrini2002detecting, beretta2017size}. These papers study the size estimates of an obstacle for the conductivity problem and the stationary Stokes system, respectively. That is, they considered an object completely immersed in the domain, which is different to our problem though, where the \emph{object} $D$ is the difference of two different bottoms and is at the solid boundary accordingly. Moreover, we are considering a homogeneous Neumann condition on the solid walls, which is natural in this context of water waves, and we are including the Dirichlet and Neumann data cases simultaneously on the free surface or, on an open subset of the free surface.  It is worth to mention that the estimation of inner cavities can be obtained as a particular case in our method.

Finally, since interactions between the free surface and the solid walls at the contact line is not well understood \cite{graham1983new,kim2015capillary}, and for avoiding cusps in the domain, we assume the free--end boundary condition, that the contact line can move vertically with a contact angle $\pi/2$.

The paper is structured as follows. In section 2, we provide some notation and preliminary results to be used along the sections that follow. In section 3 we deal with the main case of bottom detection, the case I. We state theorems \ref{th1} and \ref{th2} corresponding to the lower and upper cavity estimates, in terms of the Neumann and Dirichlet data,  simultaneously. In section 4 we study the more general case of bottom intersection. In this case we obtain a sort of logarithmic upper bound in terms of the data, which is expected from the literature. In section 5, we deal with the problem of partial measurements on the free surface for both cases I and II. In this case we obtain only an upper estimate. Finally, in section 6 we present some discussions and open problems related.

\section{Preliminaries}

In this section we introduce some definitions and previous results we will use throughout the paper.
Given $x\in\R^{N+1}$, we denote by $B_{r}(x)$ the open ball in $\R^{N+1}$, with center in $x$ and radius $r$. Also, $B_{r}'(0)$ denotes the ball in $\R^{N}$. We set $x=(x_1,\ldots,x_{N+1})$ as $x=(x',x_{N+1})$, where $x'=(x_1,\ldots,x_{N})$.

\begin{definition}{\cite[Definition 2.1]{alessandrini2002detecting}}
Let $\Omega\subset\R^{N+1}$ be a bounded domain. We shall say that $\partial\Omega$ is of class $C^{k,\alpha},$  with constants $r_0,\;M_0>0$, where $k$ is a nonnegative
integer and $\alpha\in[0,1)$, if, for any $x_0\in\partial\Omega,$ there exists a rigid transformation of coordinates, in which $x_0=0$ and
\begin{equation*}
\Omega\cap B_{r_0}(0)=\{x\in B_{r_0}(0):\;x_{n}>\varphi(x')\},
\end{equation*}
where $\varphi$ is a function of class $C^{k,\alpha}(B'_r(0))$, such that
\begin{equation*}
\begin{array}{l}
\varphi(0) =0,\\
\nabla\varphi(0)=0,\textrm{ if }k\geq 1, \\
\|\varphi\|_{C^{k,\alpha}(B'_{r_0}(0))}\leq M_0r_0.
\end{array}
\end{equation*}
\end{definition}
When $k=0$ and $\alpha=1$ we will say that $\partial\Omega$ is of Lipschitz class with constants $r_0,M_0$.

\begin{remark}
We normalize all norms in such a way that they are dimensionally equivalent to their argument, and coincide with the usual norms when $r_0=1$.
In this setup, the norm taken in the previous definition is intended as follows:
\begin{equation*}
\|\phi\|_{C^{k,\alpha}(B'_{r_0}(0))}=\sum_{i=0}^{k}r_0^{i}\|D^{i}\phi\|_{L^{\infty}(B_{r_0}'(0))}+r_0^{k+\alpha}|D^{k}\phi|_{\alpha,B_{r_0}'(0)},
\end{equation*}
where $|\cdot|$ represents the $\alpha$--H\"older seminorm
\begin{equation*}
|D^{k}\phi|_{\alpha,B_{r_0}'(0)}=\sup_{x',y'\in B_{r_0}'(0),x'\neq y'}\frac{|D^{k}\phi(x')-D^{k}\phi(y')|}{|x'-y'|^{\alpha}},
\end{equation*}
and $D^{k}\phi=\{D^{\beta}\phi\}_{|\beta|=k}$ is the set of derivatives of order $k$. Similarly we set the norms
\begin{equation*}
\begin{array}{rl}
\|u\|_{L^2(\Omega)}^2&=\displaystyle\frac{1}{r_0^{N+1}}\int_{\Omega}|u|^2,\\ \\
\|u\|_{H^1(\Omega)}^2&=\displaystyle \frac{1}{r_0^{N+1}}\left(\int_{\Omega}|u|^2+r_0^2\int_{\Omega}|\nabla u|^2\right).
\end{array}
\end{equation*}
\end{remark}

Next, we shall give some a--priori information concerning the domain $\Omega$ and the subdomain $D$ enclosed by the bottoms.
\begin{enumerate}
\item[{\bf(H1)}] We consider $\Omega\subset\R^{N+1}$ a bounded domain with connected boundary $\partial\Omega\in C^{1,1}$  with constants $r_0,M_0$. Further, there exists $M_1>0$ such that
\begin{equation}\label{Lipsch}
|\Omega|\leq M_1 r_{0}^{N+1}.
\end{equation}

\item[{\bf(H2)}] We consider $D\subset\Omega$ such that $\Omega\setminus D$ is connected and $D$ has a connected boundary $\partial D$ of Lipschitz class with constants $r,L$.

\item[{\bf(H3)}] $D$ satisfies {\bf (H2)} and the scale--invariant fatness condition with constant $Q>0$, that is
\begin{equation}\label{scale}
diam(D)\leq Qr.
\end{equation}

\item[{\bf(H4)}] Finally, we will assume that there exists a constant $h>0$, such that the \emph{fatness condition} holds, namely
\begin{equation}\label{fatness}
|D_{h}|\geq \frac 12 |D|,
\end{equation}
where we set for any $A\subset \R^{N+1}$ and $h>0$,
\begin{equation*}
A_{h}=\{x\in A:\; d(x,\partial A)>h\}.
\end{equation*}
\end{enumerate}

Let us mention some remarks about these hypotheses.
\begin{remark}
\begin{enumerate}
\item Concerning {\bf (H2)}, as explained in \cite{alessandrini2002detecting, beretta2017size}, the constant $r$ already incorporates information about the size of $D$. In addition, they showed that if $D$ has a Lipschitz boundary class with constants $r,L$, then
\begin{equation*}
|D|\geq C(L)r^{N+1}.
\end{equation*}
Besides, if condition {\bf (H3)} holds, then
\begin{equation*}
|D|\leq C(Q)r^{N+1}.
\end{equation*}
For that reason, as in \cite{alessandrini2002detecting, beretta2017size}, it will be necessary to consider $r$ as an unknown parameter.
\item Assumption {\bf (H4)} is classical in the context of size estimates (see for instance \cite{alessandrini2002detecting, beretta2017size, morassi2007size}). Moreover, if $D$ has boundary of class $C^{1,\alpha}$, it was proven in \cite{rosset1998inverse} that there exists a constant $h_1>0$ such that the condition (\ref{fatness}) holds. 
\end{enumerate}
\end{remark}

Additionallity, we shall consider the following Poincar\'e type inequality, which can be found in \cite{alessandrini2002detecting}.

\begin{proposition}{\cite[Proposition 3.2]{alessandrini2002detecting}}\label{poincare}
Let $D$ be a bounded domain in $\R^{N+1}$ of Lipschitz class with constants $r,L$ and satisfying condition (\ref{scale}) with constant $Q$. For every $u\in H^1(D)$ we have
\begin{eqnarray}
\int_{\partial D}|u-u_{\partial D}|^{2}\leq \overline{C_1}r\int_{D}|\nabla u|^{2},\label{23}\\
\int_{D}|u-u_{D}|^{2}\leq \overline{C_2}r^2\int_{D}|\nabla u|^{2},\label{23bis}
\end{eqnarray}
where
\begin{equation*}
u_{\partial D}=\frac{1}{|\partial D|}\int_{\partial D}u, \qquad u_{D}=\frac{1}{|D|}\int_{D}u,
\end{equation*}
and the constants $\overline{C_1},\overline{C_2}>0$ depend only on $L,Q$.
\end{proposition}

One important result when we are working in geometric inverse problems with one measurement is the following proposition (see \cite{alessandrini2003size}), known as \emph{Lipschitz propagation of smallness}.

\begin{proposition}{\cite[Lemma 2.2]{alessandrini2000optimal}}\label{pro1}
Let $\Omega$ be a bounded domain in $\R^{N+1}$, such that $\partial\Omega\in C^{1,1}$ with constants $r_0,M_0$.  Let $u\in H^1(\Omega)$ be the solution of the Neumann problem
\begin{equation}\label{10}
\left\{
\begin{array}{rll}
-\Delta u=0, &&\Omega,\\
\nabla u\cdot n=g, &&\partial\Omega.
\end{array}
\right.
\end{equation}
Then, for every $\rho>0$ and $x\in \Omega_{4\rho}$, we have
\begin{equation}\label{11}
\int_{B_{\rho}(x)}|\nabla u|^{2}dx\geq C_{\rho}\int_{\Omega}|\nabla u|^{2}dx,
\end{equation}
where the constant $C_{\rho}>0$ depends only on $|\Omega|$, $r_0$, $M_0$, $\frac{\|g\|_{L^2(\partial\Omega)}}{\|g\|_{H^{-1/2}(\partial\Omega)}}$, and $\rho$.
\end{proposition}

Finally, we need the following stability estimate related to ill--posed Cauchy problems for the Laplace equation in domains with $C^{1,1}$ boundary. 

\begin{proposition}{\cite[Corollary 2.1]{bourgeois2010stability}}\label{estabilidad}
Let $\Omega$ be a bounded and connected domain $\Omega\subset \R^{N+1}$ with a $C^{1,1}$ boundary $\partial\Omega$. If $\Gamma_0$ is a nonempty open set of $\partial\Omega$, then for all $k\in (0,1)$,  there exists  $C,\delta_0>0 $ such that for every $\delta\in (0,\delta_0)$ and for any $u\in H^2(\Omega)$ with $\Delta u=0$ and
\begin{equation*}
 \|u\|_{H^2(\Omega)}\leq M, \quad \|u\|_{H^1(\Gamma_0)} + \|\partial_n u\|_{L^2(\Gamma_0)}\leq \delta,
\end{equation*}
with $M$ a positive constant, we have
\begin{equation}\label{estability-estimate}
\|u\|_{H^1(\Omega)}\leq C M\left[\log\left(\frac{M}{\|u\|_{H^1(\Gamma_0)} + \|\partial_n u\|_{L^2(\Gamma_0)}}\right)\right]^{-k}.
\end{equation}
\end{proposition}

\section{Case I}\label{s_DN}

Now, we will study our first case.
Let $\Omega(b_0,\zeta_0)$ be the domain such that $\partial\Omega=\Gamma(\zeta_0)\cup\Gamma_w(b_0,\zeta_0)\cup\Gamma(b_0)$ and  $\Omega(b_1,\zeta_0)$ the domain with boundary $\partial\Omega(b_1,\zeta_0)=\Gamma(\zeta_0)\cup\Gamma_w(b_1,\zeta_0)\cup\Gamma(b_1)$. We will assume that $b_1\geq b_0$. Now, we consider $\phi_0,\phi$ be the unique weak solutions of the following problems
\begin{equation}\label{caseI-1}
\!\!\!\!\!\!\!\!\!\!\!\!\! \left\{
\begin{array}{rll}
\Delta\phi_0=0, &&\Omega(b_0,\zeta_0), \\
\phi_0=\psi_0, &&\Gamma(\zeta_0), \\
\partial_n\phi_0=0, &&\Gamma(b_0)\cup\Gamma_w(b_0,\zeta_0),
\end{array}
\right.
\qquad
\left\{
\begin{array}{rll}
\Delta\phi=0, &&\Omega(b_1,\zeta_0), \\
\phi=\psi, &&\Gamma(\zeta_0), \\
\partial_n\phi=0, & &\Gamma(b_1)\cup\Gamma_w(b_1,\zeta_0).
\end{array}
\right.
\end{equation}

Since we are performing the measurements on the common free surface given by $\Gamma(\zeta_0)$, for differentiate it we denote by $\Gamma_{up}$ this surface. We observe that the two bottoms generate a subdomain $D$, given by $D=\Omega(b_0,\zeta_0)\setminus\Omega(b_1,\zeta_0)$. Therefore, denoting by $\Omega:=\Omega(b_0,\zeta_0)$ we can rewrite systems in (\ref{caseI-1}) as (see Figure \ref{Fig1}):
\begin{equation}\label{caseI-2}
\left\{
\begin{array}{rll}
\Delta\phi_0=0, &&\Omega, \\
\phi_0=\psi_0, &&\Gamma_{up}, \\
\partial_n\phi_0=0, &&\Gamma(b_0)\cup\Gamma_w(b_0,\zeta_0),
\end{array}
\right.
\end{equation}
\begin{equation}\label{caseI-3}
\left\{
\begin{array}{rll}
\Delta\phi=0, &&\Omega\setminus D, \\
\phi=\psi, &&\Gamma_{up}, \\
\partial_n\phi=0, &&\Gamma(b_1)\cup\Gamma_w(b_1,\zeta_0),
\end{array}
\right.
\end{equation}

%Let the water-waves associated elliptic system for a bounded domain having side walls. That is, given $L>0$, if we define $\Gamma_{up}=\{(x,y)\in\R^{N}\times\R :-L\le x\le L, \ y=\zeta(x)\}$, $\Gamma_b=\{(x,y)\in\R^{N}\times \R :-L\le x\le L, \ y=b(x)\}$ and $\Gamma_w=\{(x,y)\in\R^{N}\times \R :x=-L \ \text{or} \ x=L\}$ then we assume that $\partial\Omega=\Gamma_{up}\cup\Gamma_b\cup\Gamma_w$. Let $\phi_0$ be the unique weak solution of the following problem
%\begin{equation}\label{ec1}
%\begin{cases}
%\Delta\phi_0=0, \quad &\Omega, \\
%\phi_0=\psi_0, \quad &\Gamma_{up}, \\
%\partial_n\phi_0=0, \quad &\Gamma_b\cup\Gamma_w.
%\end{cases}
%\end{equation}

%
%Next, we consider a subdomain of $\Omega$. We denote by $D\subset\Omega$ the set with boundary $\partial D=\Gamma_{d}\cup\Gamma_{b}$ as in Figure \ref{Fig1}. We assume that the boundaries $\Gamma_b$ and $\Gamma_d$ are intersected only at the sides. Therefore, let $\phi$ be the unique weak solution of
%\begin{equation}\label{ec2}
%\begin{cases}
%\Delta\phi=0, \quad &\Omega\setminus \overline{D}, \\
%\phi=\psi, \quad &\Gamma_{up}, \\
%\partial_n\phi=0, \quad &\Gamma_d\cup\Gamma_w.
%\end{cases}
%\end{equation}

	\begin{figure}[h!]
	\centering
		\includegraphics[scale=1.0]{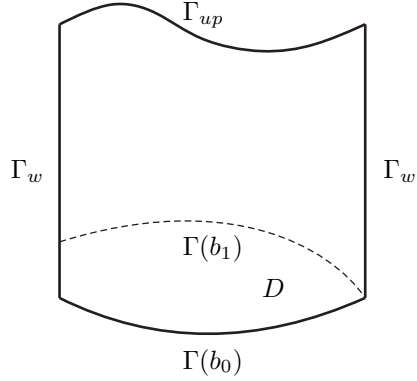}
		\put(-70,120){$\Gamma_{up}$}
		\put(-135,60){$\Gamma_w$}
		\put(6,60){$\Gamma_w$}
		\put(-70,-12){$\Gamma(b_0)$}
		\put(-70,31){$\Gamma(b_1)$}
		\put(-40,15){$D$}
	\caption{The two domains to compare the bottom and the Dirichlet and Neumann data on the free surface.}
	\label{Fig1}
	\end{figure}

%Our aim is, with one measurement on the free surface $\Gamma_{up}$,  to determine the total variation of the bottom. That is, the difference between $\Gamma(b_0)$ and $\Gamma(b_1)$. Which is equivalent to determine or estimate the volume of the unknown cavity $D$. 
%

Now, we are in position to state and prove the following identity which plays  a fundamental role in the proof of our main results. 

\begin{lemma}\label{lemma1} Let $\phi_0\in H^2(\Omega)$ and $\phi\in H^2(\Omega\setminus D)$ be the solutions of problems (\ref{caseI-2}) and (\ref{caseI-3}), respectively. Then, we have the following identities:
\begin{equation}\label{lem1-0}
\int_{\Omega\setminus D}|\nabla(\phi-\phi_0)|^2+\int_D|\nabla\phi_0|^2=\int_{\Gamma_{up}}\partial_n\phi(\psi-\psi_0)+\int_{\Gamma_{up}}(\partial_n\phi_0-\partial_n\phi)\psi_0,
\end{equation}
and
\begin{equation}\label{lem1}
2\int_{\Gamma(b_1)}\partial_n\phi_0\phi=\int_{\Gamma_{up}}(\partial_n\phi-\partial_n\phi_0)(\psi_0+\psi)+\int_{\Gamma_{up}}(\partial_n\phi+\partial_n\phi_0)(\psi_0-\psi).
\end{equation}
\end{lemma}

\begin{proof}
By multiplying the first equation of (\ref{caseI-3}) by $\phi_0$ and $\phi$, integrating in $\Omega\setminus D$ and using the boundary conditions accordingly, we obtain
\begin{equation}
\label{i1}
\int_{\Omega\setminus D}\nabla\phi\cdot\nabla\phi_0=\int_{\Gamma_{up}}\partial_n\phi\psi_0,
\end{equation}
and
\begin{equation}
\label{i2}
\int_{\Omega\setminus D}|\nabla\phi|^2=\int_{\Gamma_{up}}\partial_n\phi\psi,
\end{equation}
respectively.

Moreover, integrating again by parts in (\ref{i1}):
\begin{equation*}
\label{i3}
\int_{\Gamma_{up}}\partial_n\phi_0\psi+\int_{\Gamma(b_1)}\partial_n\phi_0\phi=\int_{\Gamma_{up}}\partial_n\phi\psi_0,
\end{equation*}
or
\begin{equation*}
\label{i4}
\int_{\Gamma(b_1)}\partial_n\phi_0\phi=\int_{\Gamma_{up}}\partial_n\phi\psi_0-\int_{\Gamma_{up}}\partial_n\phi_0\psi.
\end{equation*}

From this last equality we can get the following equivalent identities,
\begin{equation}
\label{i5}
\int_{\Gamma(b_1)}\partial_n\phi_0\phi=\int_{\Gamma_{up}}(\partial_n\phi-\partial_n\phi_0)\psi_0+\int_{\Gamma_{up}}\partial_n\phi_0(\psi_0-\psi),
\end{equation}
and
\begin{equation}
\label{i6}
\int_{\Gamma(b_1)}\partial_n\phi_0\phi=\int_{\Gamma_{up}}(\partial_n\phi-\partial_n\phi_0)\psi+\int_{\Gamma_{up}}\partial_n\phi(\psi_0-\psi).
\end{equation}

In a similar way, multiplying the first equation of (\ref{caseI-2}) by $\phi$ and integrating in $\Omega\setminus D$
\begin{equation*}
\label{i7}
\int_{\Omega\setminus D}\nabla\phi_0\cdot\nabla\phi=\int_{\Gamma_{up}}\partial_n\phi_0\psi+\int_{\Gamma(b_1)}\partial_n\phi_0\phi.
\end{equation*}

If we multiply by $\phi_0$ instead, we obtain
\begin{equation}
\label{i8}
\int_\Omega|\nabla\phi_0|^2=\int_{\Gamma_{up}}\partial_n\phi_0\psi_0.
\end{equation}

Moreover, from (\ref{i1}), (\ref{i2}), and (\ref{i8}) we get 
\begin{eqnarray*}
\int_{\Omega\setminus D}|\nabla(\phi-\phi_0)|^2+\int_D|\nabla\phi_0|^2&=\int_{\Omega\setminus D}|\nabla\phi|^2-2\int_{\Omega\setminus D}\nabla\phi_0\cdot\nabla\phi+\int_\Omega|\nabla\phi_0|^2 \\
&=\int_{\Gamma_{up}}\partial_n\phi\psi-2\int_{\Gamma_{up}}\partial_n\phi\psi_0+\int_{\Gamma_{up}}\partial_n\phi_0\psi_0 \\
&=\int_{\Gamma_{up}}\partial_n\phi(\psi-\psi_0)+\int_{\Gamma_{up}}(\partial_n\phi_0-\partial_n\phi)\psi_0.
\end{eqnarray*}

Notice that, from equations (\ref{i5}) and (\ref{i6}), one obtains the identity
\begin{equation*}
2\int_{\Gamma(b_1)}\partial_n\phi_0\phi=\int_{\Gamma_{up}}(\partial_n\phi-\partial_n\phi_0)(\psi_0+\psi)+\int_{\Gamma_{up}}(\partial_n\phi+\partial_n\phi_0)(\psi_0-\psi).
\end{equation*}
\end{proof}

Next, we state and prove our main results concerning the determination of the total variation of the bottom.

The next theorem proves the lower bound for the size of cavity $D$. It is similar to that obtained in \cite{alessandrini2002detecting} for the electrostatic potential. However, we have considered the Neumann and Dirichlet measurements simultaneously and the presence of a cavity at the boundary.

\begin{theorem}\label{th1}
Let $\phi_0\in H^2(\Omega)$ and $\phi\in H^2(\Omega\setminus D)$ be the solutions of problems (\ref{caseI-2}) and (\ref{caseI-3}), respectively. In addition, assume that $\Omega$ satisfies {\bf (H1)} and the subdomain $D$ satisfies {\bf (H3)}. 
Then, there exists a positive constant $C_1>0$ depending only on $\Omega,L,Q$  such that
\begin{equation}\label{the1}
\frac{\left(\int_{\Gamma_{up}}(\partial_n\phi-\partial_n\phi_0)\psi-\int_{\Gamma_{up}}\partial_n\phi(\psi-\psi_0)\right)^2}{\int_{\Gamma_{up}}\psi_0\partial_n\phi_0\int_{\Gamma_{up}}\psi\partial_n\phi}\le C_1|D|.
\end{equation}
\end{theorem}

\begin{proof}
Note first that integrating the first equation of (\ref{caseI-2}) on $\Omega$ and applying the divergence theorem, we obtain
\begin{equation*}
%\label{e1}
0=\int_{\Gamma_{up}}\partial_n\phi_0.
\end{equation*}

Thus, if we denote by $\bar{\phi}=\int_{\Gamma(b_1)}\phi$, from identity (\ref{i6}) and H\"older's inequality we get
\begin{eqnarray}
\label{e2}
\int_{\Gamma_{up}}(\partial_n\phi-\partial_n\phi_0)\psi-\int_{\Gamma_{up}}\partial_n\phi(\psi-\psi_0)=\int_{\Gamma(b_1)}\partial_n\phi_0\phi \nonumber\\
=\int_{\Gamma(b_1)}\partial_n\phi_0(\phi-\bar{\phi}) 
\le \left(\int_{\Gamma(b_1)}(\partial_n\phi_0)^2\right)^{1/2}\left(\int_{\Gamma(b_1)}(\phi-\bar{\phi})^2\right)^{1/2}.
\end{eqnarray}

First, from the Poincar\'e type inequality (\ref{23}), there exists a constant $C>0$ such that the second term in the right hand side of (\ref{e2}) can be estimated as
\begin{equation}
\label{e3}
\int_{\Gamma(b_1)}(\phi-\bar{\phi})^2\le C\int_{\Omega\setminus D}|\nabla\phi|^2 
=C\int_{\Gamma_{up}}\psi\partial_n\phi.
\end{equation}

Second, we need an estimate that allows for control the normal derivative of $\phi_0$, by the gradient of the function. Namely, 
\begin{equation}
\label{e4}
\int_{\Gamma(b_1)}(\partial_n\phi_0)^2\le C|D|\int_{\Gamma_{up}}\psi_0\partial_n\phi_0.
\end{equation}

To prove (\ref{e4}), let $\phi_0\in H^2(\Omega)$. Since $D$ is piecewise $C^1$, by the trace theorem
\begin{equation}
\label{e5}
\int_{\Gamma(b_1)}(\partial_n\phi_0)^2=\int_{\Gamma(b_1)}(\partial_n(\phi_0-\bar{\phi}_0))^2 \le C\|\nabla(\phi_0-\bar{\phi}_0)\|_{H^1(D)}^2.
\end{equation}

Now, since $\phi_0$ satisfy $\Delta \phi_0=0$ in $\Omega$ and $D$ is bounded we obtain 
\begin{equation}
\label{e6}
\|\nabla(\phi_0-\bar{\phi}_0)\|_{H^1(D)}^2\le C|D|\sup_D|\nabla(\phi_0-\bar{\phi}_0)|^2.
\end{equation}

To end the proof, we need to use some classical elliptic estimates as well as some interior estimates as in \cite{alessandrini2002detecting}. Since $\partial D\cap\partial\Omega\ne\emptyset$, the way to transform $D$ an inner set is the following: first, extend $\Omega$ downward, by symmetry. Then the new domain can be extended, one time, rightward and leftward by a harmonic function. This can be done thanks to the zero Neumann boundary condition satisfied by $\phi_0$ in $\Omega$. To avoid cusps at the top of the boundary, $\Gamma_{up}$, we assume the free-end boundary condition for inviscid fluids in contact with solid walls, that the contact line can move vertically with a contact angle $\pi/2$ \cite{benjamin1979gravity,graham1983new}.

Therefore, we obtain a new larger domain $\widetilde{\Omega}$ such that $D\subset\subset \widetilde{\Omega}$, as in Figure \ref{Fig4}. 

In this case, we have that there exists a positive constant $d_0>0$ such that 
\begin{equation}\label{H2}
d(D,\partial\widetilde{\Omega})\geq d_0 >0.
\end{equation}
Let $\widetilde{\Omega}_{\frac{d_0}{2}}$ be an intermediate domain. Recalling that $d(D,\partial\widetilde{\Omega})\geq d_0$, we have $d(D,\partial\widetilde{\Omega}_{d_{0}/2})\geq\frac{d_0}{2}.$ 
In addition, we extend the solution $\phi_0$ of (\ref{caseI-2}) by $\widetilde{\phi}_0$ to the new domain $\widetilde{\Omega}$, to obtain:
\begin{equation}\label{ec1.1}
\left\{
\begin{array}{rll}
\Delta\widetilde{\phi}_0=0, &&\widetilde{\Omega}\setminus D, \\
\widetilde{\phi}_0=\widetilde{\psi}_0, &&\widetilde{\Gamma}_{up}\cup\widetilde{\Gamma}(b_0), \\
\partial_n\widetilde{\phi}_0=0, &&\widetilde{\Gamma}(b_0)\cup\widetilde{\Gamma}_w(b_0,\zeta_0),
\end{array}
\right.
\end{equation}
with $\widetilde{\Gamma}_{up}$, $\widetilde{\Gamma}(b_0)$, and $\widetilde{\Gamma}_w(b_0,\zeta_0)$ being the corresponding extensions of $\Gamma_{up}$, $\Gamma(b_0)$, and $\Gamma_w(b_0,\zeta_0)$, respectively.

	\begin{figure}[h!]
	\centering
		\includegraphics[scale=0.7]{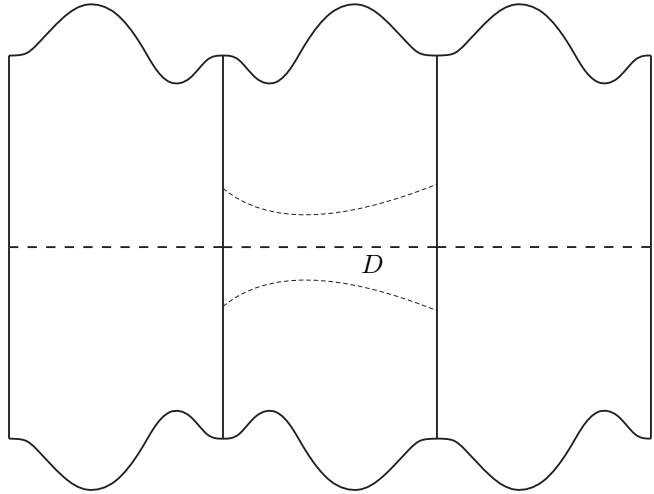}
		\put(-110,83){$D$}
		\caption{The larger domain by reflection to make $D$ an inner set.}
		\label{Fig4}
	\end{figure}

Following the ideas in \cite[Page 61]{alessandrini2000optimal} we have
\begin{equation}
\label{e7}
\sup_D|\nabla(\widetilde{\phi}_0-\bar{\phi}_0)|^2 \le C\sup_{\widetilde{\Omega}_{\frac{d_0}{2}}}|\widetilde{\phi}_0-\bar{\phi}_0|^2\le C\|\widetilde{\phi}_0-\bar{\phi}_0\|_{L^2(\widetilde{\Omega})}^2.
\end{equation}

Thus, by Poincar\'e's inequality (\ref{23bis}), we have
\begin{equation}
\label{e8}
\int_{\widetilde{\Omega}}|\widetilde{\phi}_0-\bar{\phi}_0|^2 \le C\int_{\widetilde{\Omega}}|\nabla\widetilde{\phi}_0|^2 =C\int_{\Gamma_{up}}\psi_0\partial_n\phi_0.
\end{equation}

Then, from (\ref{e5}) and (\ref{e8}), we get (\ref{e4}). Putting together (\ref{e2}) and (\ref{e4}) we obtain the estimate 
\begin{equation*}
\frac{\left(\int_{\Gamma_{up}}(\partial_n\phi-\partial_n\phi_0)\psi-\int_{\Gamma_{up}}\partial_n\phi(\psi-\psi_0)\right)^2}{\int_{\Gamma_{up}}\psi_0\partial_n\phi_0\int_{\Gamma_{up}}\psi\partial_n\phi}\le C|D|,
\end{equation*}
and the proof finishes.
\end{proof}

Our second main result of this section is the following upper bound of the total variation of the bottom. As before, this estimate is proved for cavities at the boundary and includes both, Dirichlet and Neumann measurements, at the free surface.

\begin{theorem}\label{th2}
Let $\phi_0\in H^2(\Omega)$ and $\phi_1\in H^2(\Omega\setminus D)$ be the solutions of problems (\ref{caseI-2}) and (\ref{caseI-3}), respectively. Assume that $\Omega$ satisfies {\bf (H1)} and the subdomain $D$ satisfies {\bf (H2)} and  {\bf (H4)}.
Then, there exists a constant $C_2>0$ such that 
\begin{equation}\label{the2}
\frac{\int_{\Gamma_{up}}\partial_n\phi(\psi-\psi_0)+\int_{\Gamma_{up}}(\partial_n\phi_0-\partial_n\phi)\psi_0}{\int_{\Gamma_{up}}\psi_0\partial_n\phi_0}\ge C_2|D|.
\end{equation}
The constant $C_2$ depends only on $|\Omega|$, $r_0$, $M_0$, $\frac{\|\partial_n\phi_0\|_{L^2(\Gamma_{up})}}{\|\partial_n\phi_0\|_{H^{-1/2}(\Gamma_{up})}}$.
\end{theorem}

\begin{proof}
We observe that, by (\ref{lem1-0}), we have
\begin{equation}
\label{e10}
\int_D|\nabla\phi_0|^2\le\int_{\Gamma_{up}}\partial_n\phi(\psi-\psi_0)+\int_{\Gamma_{up}}(\partial_n\phi_0-\partial_n\phi)\psi_0.
\end{equation}

Based on Alessandrini \cite{alessandrini2000optimal}, we proceed as follows. Let $D_h\subset\subset D$ such that $|D_h|\ge\frac{1}{2}|D|$ and consider $Q_\alpha$ a uniform mesh of boxes $q$ of side $\varepsilon$ for $D_h$, such that $D_h\subset\bigcup Q_\alpha\subset D$. Then, 
\begin{equation}
\label{e12}
\int_D|\nabla\phi_0|^2\ge \sum_{q\cap D_h\ne\emptyset}\int_q|\nabla\phi_0|^2\ge\frac{\int_{q_0}|\nabla\phi_0|^2}{|q_0|}\sum_{q\cap D_h\ne\emptyset}|q|
=\frac{|D_h|}{|q_0|}\int_{q_0}|\nabla\phi_0|^2,
\end{equation}
where $q_0\subset Q_\alpha$ is such that
\begin{equation*}
\min_{q}\int_q|\nabla\phi_0|^2=\int_{q_0}|\nabla\phi_0|^2.
\end{equation*}
Since $\phi_0$ is the unique weak solution of (\ref{caseI-2}), we obtain that the previous minimum is strictly positive.

Let $x_0$ be the center of $q_0$. Using the estimate (\ref{11}) in Proposition \ref{pro1} with $x=x_0$ and $r=\frac{\varepsilon}{2}$, we have that
\begin{equation}\label{e12.1}
\int_{q_0}|\nabla\phi_0|^2\geq C(|q_0|)\int_\Omega|\nabla\phi_0|^2.
\end{equation}

Therefore, replacing (\ref{e12.1}) in (\ref{e12}) and using the fatness condition (\ref{fatness}), we obtain 
\begin{equation*}
\int_D|\nabla\phi_0|^2\ge\frac{|D|}{2|q_0|}C(|q_0|)\int_{\Gamma_{up}}\psi_0\partial_n\phi_0.
\end{equation*}

Then, from (\ref{e10}) we deduce the desired result
\begin{equation*}
\frac{\int_{\Gamma_{up}}\partial_n\phi(\psi-\psi_0)+\int_{\Gamma_{up}}(\partial_n\phi_0-\partial_n\phi)\psi_0}{\int_{\Gamma_{up}}\psi_0\partial_n\phi_0}\ge C(|q_0|)|D|.
\end{equation*}
\end{proof}

Let us make the following comments about the results obtained so far.
\begin{remark}
\begin{enumerate}
\item Estimates in Theorems \ref{th1} and \ref{th2} are, in some sense, more general to those from the literature. For instance, if the Neumann measurements of $\phi$ and $\phi_0$ are equal in $\Gamma_{up}$, then (\ref{the1}) and (\ref{the2}) become
\begin{equation}\label{remark1}
C_1\frac{\left(\int_{\Gamma_{up}}\partial_n\phi(\psi-\psi_0)\right)^2}{\int_{\Gamma_{up}}\psi_0\partial_n\phi_0\int_{\Gamma_{up}}\psi\partial_n\phi}\leq |D|\leq C_2\frac{\int_{\Gamma_{up}}\partial_n\phi(\psi-\psi_0)}{\int_{\Gamma_{up}}\psi_0\partial_n\phi_0},
\end{equation}
which corresponds to those in \cite{alessandrini2003size, alessandrini2000optimal, beretta2017size}.
\item In the more general case when the bottoms $b_0,b_1$ are such that $b_1(x)\ge b_0(x)$ for any $x\in S$, and $D=\cup_{i=1}^nD_i$ is such that $\Omega\setminus D$ is connected, estimates in Theorems \ref{th1} and \ref{th2} also holds.
\end{enumerate}
\end{remark}

\section{Case II}\label{sec3}

If we consider now the possibility of finite intersections between the bottoms, estimates above are still true and something similar is obtained. Let us change notation slightly and consider $\phi_1$ and $\phi_2$ defined on $\Omega_1:=\Omega(b_1,\zeta_0)$ and  $\Omega_2:=\Omega(b_2,\zeta_0)$, respectively. In this case we assume that, near to the bottom, $\Omega_1\cap\Omega_2\neq \emptyset$ and we are still measuring on the free common surface region $\Gamma_{up}:=\Gamma(\zeta_0)$. Without loss of generality, we will assume that the bottom intersections are at least one and at most three. Then, if we define $\Gamma(b_1)=:\Gamma_b^1\cup\Gamma_d^2$,  $\Gamma(b_2)=:\Gamma_d^1\cup\Gamma_b^2$ as in Figure \ref{Fig3}, we consider the following two problems
\begin{equation}
\label{bi1}
\left\{
\begin{array}{rll}
\Delta\phi_1=0, &&\Omega_1, \\
\phi_1=\psi_1, &&\Gamma_{up}, \\
\partial_n\phi_1=0, &&\Gamma(b_1)\cup\Gamma_w(b_1,\zeta_0),
\end{array}
\right.
\;
\left\{
\begin{array}{rll}
\Delta\phi_2=0, &&\Omega_2, \\
\phi_2=\psi_2, &&\Gamma_{up}, \\
\partial_n\phi_2=0, &&\Gamma(b_2)\cup \Gamma_w(b_2,\zeta_0).
\end{array}
\right.
\end{equation}

	\begin{figure}[h!]
	\centering
		\includegraphics[scale=1.0]{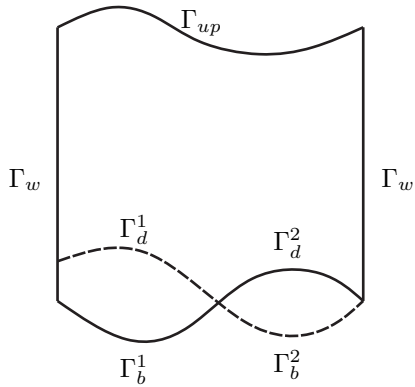}
		\put(-70,120){$\Gamma_{up}$}
		\put(-135,60){$\Gamma_w$}
		\put(6,60){$\Gamma_w$}
		\put(-93,41){$\Gamma_d^1$}
		\put(-93,-12){$\Gamma_b^1$}
		\put(-35,36){$\Gamma_d^2$}
		\put(-35,-10){$\Gamma_b^2$}
		\caption{The two domains when bottoms are intersected and the measurements are performed on a common free surface $\Gamma_{up}$.}
		\label{Fig3}
	\end{figure}

We start with the lower bound of the set $\Omega_1\vartriangle\Omega_2$. In this case, we will use some arguments from the previous sections. The first main result is a restatement of Theorem \ref{th1} above, where we present the lower bound arising from the new geometric assumptions.

\begin{theorem}
Assume that $\Omega_1,\Omega_2$ satisfy {\bf (H1)} with constants $L_1,Q_1$ and $L_2,Q_2$, respectively,  and $\Omega_1\vartriangle \Omega_2$ satisfies {\bf (H3)}. Then, there exists a constant $C_3>0$ such that
\begin{equation}
\label{bi11}
\frac{\left(\int_{\Gamma_{up}}\left[\partial_n\phi_2(\phi_1-\phi_2)+(\partial_n\phi_2-\partial_n\phi_1)\phi_2\right]\right)^2}{\left(\int_{\Gamma_{up}}\phi_1\partial_n\phi_1\right)^2+\left(\int_{\Gamma_{up}}\phi_2\partial_n\phi_2\right)^2}\le C_3|\Omega_1\vartriangle\Omega_2|.
\end{equation}
The constant $C_3$ depends only on $\Omega, L_1,Q_1,L_2,Q_2$.
\end{theorem}

\begin{proof}
As it was done in the proof of Theorem \ref{th1} , we have
\begin{equation*}
%\label{bi2}
\int_{\Omega_1\cap\Omega_2}\nabla\phi_1\cdot\nabla\phi_2=\int_{\partial(\Omega_1\cap\Omega_2)}\phi_1\partial_n\phi_2=\int_{\partial(\Omega_1\cap\Omega_2)}\phi_2\partial_n\phi_1,
\end{equation*}
which, by the boundary conditions, is equivalent to
\begin{equation*}
%\label{bi3}
\int_{\Gamma_{up}}(\phi_1\partial_n\phi_2-\phi_2\partial_n\phi_1)=\int_{\Gamma_d^1}\phi_2\partial_n\phi_1-\int_{\Gamma_d^2}\phi_1\partial_n\phi_2,
\end{equation*}
or
\begin{equation}
\label{bi4}
\int_{\Gamma_{up}}\left[\partial_n\phi_2(\phi_1-\phi_2)+(\partial_n\phi_2-\partial_n\phi_1)\phi_2\right]=\int_{\Gamma_d^1}\phi_2\partial_n\phi_1-\int_{\Gamma_d^2}\phi_1\partial_n\phi_2.
\end{equation}

Thus, if we consider the energy  $E$ defined by
\begin{equation*}
E=\int_{\Omega_1\cap\Omega_2}|\nabla(\phi_1-\phi_2)|^2+\int_{\Omega_1\setminus\Omega_2}|\nabla\phi_1|^2+\int_{\Omega_2\setminus\Omega_1}|\nabla\phi_2|^2, 
\end{equation*}
proceeding as we did in Section \ref{s_DN}, and using (\ref{bi4}), we obtain
\begin{equation*}
E=\int_{\Gamma_{up}}\left[\partial_n\phi_1(\phi_1-\phi_2)+(\partial_n\phi_2-\partial_n\phi_1)\phi_2\right]-2\int_{\Gamma_d^1}\phi_2\partial_n\phi_1.
\end{equation*}

Notice that, in the last identity one can replace $\int_{\Gamma_d^1}\phi_2\partial_n\phi_1$ by $\int_{\Gamma_d^2}\phi_1\partial_n\phi_2$ instead.

Moreover, using the boundary conditions we can write identity above as
\begin{equation}
\label{bi6}
E=\int_{\Gamma_{up}}\left[\partial_n\phi_1(\phi_1-\phi_2)+(\partial_n\phi_2-\partial_n\phi_1)\phi_2\right]-2\int_{\Gamma_d^1}\phi_2(\partial_n\phi_1-\partial_n\phi_2).
\end{equation}

Now, from (\ref{bi4}), and following the arguments presented in Section \ref{s_DN},
\begin{multline}
\label{bi7}
\int_{\Gamma_{up}}\left[\partial_n\phi_2(\phi_1-\phi_2)+(\partial_n\phi_2-\partial_n\phi_1)\phi_2\right] 
\le\left(\int_{\Gamma_d^1}(\partial_n\phi_1)^2\right)^{1/2}\left(\int_{\Gamma_d^1}(\phi_2-\bar{\phi}_2)^2\right)^{1/2} \\
+\left(\int_{\Gamma_d^2}(\partial_n\phi_2)^2\right)^{1/2}\left(\int_{\Gamma_d^2}(\phi_1-\bar{\phi}_1)^2\right)^{1/2}.
\end{multline}

As before, for $i=1,2$, using the Poincar\'e type inequality (\ref{23}), we obtain 
\begin{equation}
\label{bi8}
\int_{\Gamma_d^i}(\phi_i-\bar{\phi}_i)^2 \le C\int_{\Omega_1\cap\Omega_2}|\nabla\phi_i|^2\le C\int_{\Omega_i}|\nabla\phi_i|^2=C\int_{\Gamma_{up}}\phi_i\partial_n\phi_i.
\end{equation}

Moreover, for $i=1,2$, if $\phi_i\in H^2(\Omega_i)$ and since $\Omega_1\vartriangle\Omega_2$ is piecewise $C^1$, by the trace theorem, we have 
\begin{equation}
\label{bi9}
\int_{\Gamma_d^i}(\partial_n\phi_i)^2=\int_{\Gamma_d^i}(\partial_n(\phi_i-\bar{\phi}_i))^2 \le C\|\nabla(\phi_i-\bar{\phi}_i)\|_{H^1(\Omega_i\setminus\Omega_{3-i})}^2.
\end{equation}

Repeating the arguments from Section \ref{s_DN}; specifically the extension of the domain in such a way that the set $\Omega_1\vartriangle\Omega_2$ becomes an inner subset of a larger domain $\widetilde{\Omega}$, we obtain the following
\begin{align}\label{bi9.1}
\|\nabla(\phi_i-\bar{\phi}_i)\|_{H^1(\Omega_i\setminus\Omega_{3-i})}^2&\le C\|\nabla(\phi_i-\bar{\phi}_i)\|_{L^2(\Omega_i\setminus\Omega_{3-i})}^2 \nonumber\\
&\le C|\Omega_i\setminus\Omega_{3-i}|\sup_{\Omega_i\setminus\Omega_{3-i}}|\nabla(\phi_i-\bar{\phi}_i)|^2 \nonumber\\
&\le C|\Omega_i\setminus\Omega_{3-i}|\sup_{\Omega_i\setminus\Omega_{3-i}}|\phi_i-\bar{\phi}_i|^2 \nonumber\\
&\le C|\Omega_i\setminus\Omega_{3-i}|\int_{\Omega_i\setminus\Omega_{3-i}}|\phi_i-\bar{\phi}_i|^2 \nonumber\\
&\le C|\Omega_i\setminus\Omega_{3-i}|\int_{\Omega_i\setminus\Omega_{3-i}}|\nabla\phi_i|^2 \nonumber\\
&\le C|\Omega_i\setminus\Omega_{3-i}|\int_{\Omega_i}|\nabla\phi_i|^2 \nonumber\\
&=C|\Omega_i\setminus\Omega_{3-i}|\int_{\Gamma_{up}}\phi_i\partial_n\phi_i.
\end{align}

Putting together (\ref{bi7})--(\ref{bi9.1}), we obtain
\begin{align*}
%\label{bi10}
\left(\int_{\Gamma_{up}}\left[\partial_n\phi_2(\phi_1-\phi_2)+(\partial_n\phi_2-\partial_n\phi_1)\phi_2\right]\right)^2 
\le C|\Omega_1\setminus\Omega_2|\left(\int_{\Gamma_{up}}\phi_1\partial_n\phi_1\right)^2+C|\Omega_2\setminus\Omega_1|\left(\int_{\Gamma_{up}}\phi_2\partial_n\phi_2\right)^2,
\end{align*}
which implies

\begin{equation*}
\frac{\left(\int_{\Gamma_{up}}\left[\partial_n\phi_2(\phi_1-\phi_2)+(\partial_n\phi_2-\partial_n\phi_1)\phi_2\right]\right)^2}{\left(\int_{\Gamma_{up}}\phi_1\partial_n\phi_1\right)^2+\left(\int_{\Gamma_{up}}\phi_2\partial_n\phi_2\right)^2}\le C|\Omega_1\vartriangle\Omega_2|.
\end{equation*}
\end{proof}

We complement the theorem above with the following upper bound for the variation of the bottom.
\begin{theorem}\label{th4}
Assume that $\Omega_1,\Omega_2$ satisfy {\bf (H1)} with constants $L_1,Q_1$ and $L_2,Q_2$, respectively, $\Omega_1\vartriangle\Omega_2$ satisfies {\bf (H2)}, and $\Omega_1\setminus\Omega_2, \Omega_2\setminus\Omega_1$ satisfy {\bf (H4)}, with constants $r_1,M_1$ and $r_2,M_2$, respectively. Then, for all $k\in(0,1)$ and $i=1,2$, there exist $C_4,\delta_0 >0$ such that for every $\delta\in (0,\delta_0)$ and
\begin{equation*}
 \|\phi_i\|_{H^2((\Omega_1\cup\Omega_2)\setminus(\Omega_1\vartriangle\Omega_2))}\leq M, \quad \|\phi_i\|_{H^1(\Gamma_{up})} + \|\partial_n \phi_i\|_{L^2(\Gamma_{up})}\leq \delta, 
\end{equation*}
where $M>0$, we obtain
\begin{equation}\label{binew}
|\Omega_1\vartriangle\Omega_2| \le C_4 \sum_{i=1}^2\frac{\int_{\Gamma_{up}}\left[\partial_n\phi_1(\phi_1-\phi_2)+(\partial_n\phi_2-\partial_n\phi_1)\phi_2\right]+ A}{\int_{\Gamma_{up}}\phi_i\partial_n\phi_i},
\end{equation}
where $A$ is given by
\begin{equation*}
A:=\|\phi_2\|_{H^1(\Omega_2)}\left[ \log\left(\frac{M}{\|\phi_1-\phi_2\|_{H^1(\Gamma_{up})} + \|\partial_n (\phi_1-\phi_2)\|_{L^2(\Gamma_{up})}}\right)\right]^{-k},
\end{equation*}
and  the constant $C_4$ depends only on $|\Omega_i|$, $r_i$, $M_i$, $\frac{\|\partial_n\phi_i\|_{L^2(\Gamma_{up})}}{\|\partial_n\phi_i\|_{H^{-1/2}(\Gamma_{up})}}$.
\end{theorem}

\begin{proof}

We observe that  from (\ref{bi6}) we have
\begin{eqnarray}
\label{bi12}
\int_{\Omega_1\setminus\Omega_2}|\nabla\phi_1|^2+\int_{\Omega_2\setminus\Omega_1}|\nabla\phi_2|^2 
\le \int_{\Gamma_{up}}\left[\partial_n\phi_1(\phi_1-\phi_2)+(\partial_n\phi_2-\partial_n\phi_1)\phi_2\right] \nonumber\\
\qquad+2\int_{\Gamma_d^1}\phi_2|\partial_n\phi_1-\partial_n\phi_2|.
\end{eqnarray}

Now, since $\phi_i\in H^2(\Omega_i)$ and the boundaries $\partial\Omega_i$ are of class $C^{1,1}$, from the trace Theorem, the last term in the right hand side of (\ref{bi12}) can be estimated as follows:
\begin{eqnarray}\label{bi15}
\int_{\Gamma_d^1}&\phi_2|\partial_n\phi_1-\partial_n\phi_2|\le \left(\int_{\Gamma_d^1}\phi_2^2\right)^{1/2}\left(\int_{\Gamma_d^1}(\partial_n\phi_1-\partial_n\phi_2)^2\right)^{1/2} \nonumber \\
&\le \|\phi_2\|_{H^1(\Omega_2)}\left(\int_{\left((\Omega_1\cup\Omega_2)\setminus(\Omega_1\vartriangle\Omega_2)\right)}|\nabla(\phi_1-\phi_2)|^2+|\phi_1-\phi_2|^2\right)^{1/2}.
\end{eqnarray}

From the above estimate and by Proposition \ref{estabilidad} applied to the second term in the right hand side of (\ref{bi15}), we get
\begin{align}\label{bi13}
\int_{\Gamma_d^1}\phi_2|\partial_n\phi_1-\partial_n\phi_2|\leq C  \|\phi_2\|_{H^1(\Omega_2)}\left[  \log\left(\frac{M}{\|\phi_1-\phi_2\|_{H^1(\Gamma_{up})} + \|\partial_n (\phi_1-\phi_2)\|_{L^2(\Gamma_{up})}}\right)\right]^{-k}.
\end{align}

Finally, since $\Omega_1\setminus\Omega_2, \Omega_2\setminus\Omega_1$ satisfy {\bf (H4)},  following the same arguments in (\ref{e12}), for $i=1,2$ we have
\begin{equation}
\label{bi14}
\int_{\Omega_i\setminus\Omega_{3-i}}|\nabla\phi_i|^2\ge C|\Omega_i\setminus\Omega_{3-i}|\int_{\Omega_i}|\nabla\phi_i|^2 
= C|\Omega_i\setminus\Omega_{3-i}|\int_{\Gamma_{up}}\phi_i\partial_n\phi_i.
\end{equation}

Then, from (\ref{bi12})--(\ref{bi14}), we obtain the desired result and the proof is finished.
\end{proof}

\begin{remark}
	Notice that estimates (\ref{bi11}) and (\ref{binew}), are different to (\ref{the1}) and (\ref{the2}) respectively. This is mostly due to the overlapping between the bottoms. That is, if we consider two bottoms $b_0,b_1$ such that $b_1(x)-b_0(x)$, $x\in S$, switch between positive and negative a finite number of times, then identities (\ref{bi4}) and (\ref{bi6}) possess more terms when integrating by parts.
\end{remark}

\section{Case III}

In this section, using the computations of the previous cases, we bound the volume of $D$ when the measurements of the Dirichlet and Neumann data are performed on an open subset of the free--surface.

Let us start with the Case I. Following the notation introduced in Section \ref{s_DN}, we consider $\phi_0$ and $\phi$ the weak solutions of problems:
\begin{equation}\label{partial-1}
\left\{
\begin{array}{rll}
\Delta\phi_0=0, &&\Omega, \\
\phi_0=\psi_0, &&\Gamma^{*}, \\
\partial_n\phi_0=0, &&\Gamma(b_0)\cup\Gamma_w(b_0,\zeta_0),
\end{array}
\right.
\end{equation}
\begin{equation}\label{partial-2}
\left\{
\begin{array}{rll}
\Delta\phi=0, &&\Omega\setminus D, \\
\phi=\psi, &&\Gamma^{*}, \\
\partial_n\phi=0, &&\Gamma(b_1)\cup\Gamma_w(b_1,\zeta_0),
\end{array}
\right.
\end{equation}
where $\Gamma^{*}$ is an open subset of $\Gamma_{up}$.

\begin{theorem}\label{th5}
Let $\phi_0\in H^2(\Omega)$ and $\phi\in H^2(\Omega\setminus D)$ be the solutions of problems (\ref{partial-1}) and (\ref{partial-2}), respectively. Assume that $\Omega$ satisfies {\bf (H1)} and the subdomain $D$ satisfies {\bf (H2)} and  {\bf (H4)}.
Then, for all $k\in(0,1)$ there exist $C_5,\delta_0 >0$ such that for every $\delta\in (0,\delta_0)$, and
\begin{eqnarray*}
 \|\phi_0\|_{H^2(\Omega)}\leq M, \quad \|\phi_0\|_{H^1(\Gamma_{up})} + \|\partial_n \phi_0\|_{L^2(\Gamma_{up})}\leq \delta, \\
 \|\phi\|_{H^2(\Omega\setminus D)}\leq M, \quad \|\phi\|_{H^1(\Gamma_{up})} + \|\partial_n \phi\|_{L^2(\Gamma_{up})}\leq \delta,
\end{eqnarray*}
where $M>0$, we have
\begin{align*}\label{the2.1}
|D|\leq C_5\left(\|\nabla\phi\|_{H^1(\Omega\setminus D)}+\|\phi_0\|_{H^1(\Omega)}\right) \log\left(\frac{M}{\|\psi-\psi_0\|_{H^1(\Gamma^{*})}+\|\partial_n(\phi-\phi_0)\|_{L^2(\Gamma^{*})}}\right)^{-k}.
\end{align*}
The constant $C_5$ depends only on $|\Omega|$, $r_0$, $M_0$, $M$, $\frac{\|\partial_n\phi_0\|_{L^2(\Gamma_{up})}}{\|\partial_n\phi_0\|_{H^{-1/2}(\Gamma_{up})}}$.

\end{theorem}

\begin{proof}
Observe that, from Theorem \ref{th2}, the following estimate holds:
\begin{equation*}
|D|\leq C_2 \frac{\int_{\Gamma_{up}}\partial_n\phi(\phi-\phi_0)+\int_{\Gamma_{up}}(\partial_n\phi_0-\partial_n\phi)\phi_0}{\int_{\Gamma_{up}}\phi_0\partial_n\phi_0}.
\end{equation*}

From (\ref{i8}), we have that 
\begin{equation*}
0<\int_{\Gamma_{up}}\phi_0\partial_n\phi_0.
\end{equation*}

Therefore, we get
\begin{equation*}
|D|\leq \widetilde{C}_2\left(\int_{\Gamma_{up}}\partial_n\phi(\phi-\phi_0)+\int_{\Gamma_{up}}(\partial_n\phi_0-\partial_n\phi)\phi_0\right).
\end{equation*}

Using H\"older's inequality and trace theorem, we obtain that
\begin{align*}
|D|\leq \widetilde{C}_2\left[\|\nabla\phi\|_{H^1(\Omega\setminus D)}\left(\int_{\Gamma_{up}}|\phi-\phi_0|^2\right)^{1/2}+\|\phi_0\|_{H^1(\Omega)}\left(\int_{\Gamma_{up}}|\partial_n\phi_0-\partial_n\phi|^2\right)^{1/2} \right].
\end{align*}
Finally, applying Proposition \ref{estabilidad}, we obtain the desired result and the proof ends.
\end{proof}

For the Case II, we consider $\phi_1$ and $\phi_2$ be the solutions of problem (\ref{bi1}) with Dirichlet boundary data in $\Gamma^{*}$. That is,
\begin{equation*}
\label{partial-3}
\left\{
\begin{array}{rll}
\Delta\phi_1=0, &&\Omega_1, \\
\phi_1=\psi_1, &&\Gamma^{*}, \\
\partial_n\phi_1=0, &&\Gamma_b\cup\Gamma_w(b_1,\zeta_0),
\end{array}
\right.
\quad
\left\{
\begin{array}{rll}
\Delta\phi_2=0, &&\Omega_2, \\
\phi_2=\psi_2, &&\Gamma^{*}, \\
\partial_n\phi_2=0, &&\Gamma_d\cup \Gamma_w(b_2,\zeta_0).
\end{array}
\right.
\end{equation*}

Following the same arguments that in the previous theorem, we obtain the next result about the size estimate of $D$, when the measurements are performed on an open subset of $\Gamma_{up}$.

\begin{theorem}\label{th6}
Assume that $\Omega_1,\Omega_2$ satisfy {\bf (H1)} with constants $L_1,Q_1$ and $L_2,Q_2$, respectively. Assume also that $\Omega_1\vartriangle\Omega_2$ satisfies {\bf (H2)}, and $\Omega_1\setminus\Omega_2$, $\Omega_2\setminus\Omega_1$ satisfy {\bf (H4)}, with constants $r_1,M_1$ and $r_2,M_2$, respectively. Then, for all $k\in(0,1)$, and $i=1,2$, there exist $C_6,\delta_0 >0$ such that for every $\delta\in (0,\delta_0)$ and
\begin{equation*}
 \|\phi_i\|_{H^2((\Omega_1\cup\Omega_2)\setminus(\Omega_1\vartriangle\Omega_2))}\leq M, \quad \|\phi_i\|_{H^1(\Gamma_{up})} + \|\partial_n \phi_i\|_{L^2(\Gamma_{up})}\leq \delta, 
\end{equation*}
where $M$ is a positive constant, we have
\begin{align*}
%\label{bi15}
|\Omega_1\vartriangle\Omega_2| \le C_6\left( \|\phi_2\|_{H^1(\Omega_2)}+\|\nabla\phi_1\|_{H^1(\Omega_1)}\right) 
\left[\log\left(\frac{M}{\|\psi_1-\psi_2\|_{H^1(\Gamma^{*})} + \|\partial_n (\phi_1-\phi_2)\|_{L^2(\Gamma^{*})}}\right)\right]^{-k},
\end{align*}
and  the constant $C_6$ depends only on $|\Omega_i|$, $r_i$, $M_i$, $M$, $\frac{\|\partial_n\phi_i\|_{L^2(\Gamma_{up})}}{\|\partial_n\phi_i\|_{H^{-1/2}(\Gamma_{up})}}$.

\end{theorem}

%
%\begin{center}
%	\begin{figure}[h!]
%		\includegraphics[scale=1.0]{fig3}
%		\put(-136,60){$\Gamma_w$}
%			\put(6,60){$\Gamma_w$}
%			\put(-100,90){$\Gamma_{s,2}^1$}
%			\put(-100,140){$\Gamma_{s,1}^1$}
%			\put(-30,140){$\Gamma_{s,2}^2$}
%			\put(-30,90){$\Gamma_{s,1}^2$}
%			\put(-50,-12){$\Gamma_b$}
%			\put(-60,34){$\Gamma_d$}
%			\put(-50,9){$D$}
%		\caption{The two domains when bottoms are intersected and the measurements are performed on a common free surface $\Gamma_{up}$.}
%		\label{Fig7}
%	\end{figure}
%\end{center}

\section{Further comments and future work}

We have developed a method to estimate the size of a cavity along the rigid boundary through measurements on the free surface on a potential and perfect fluid. We have used the context of the water--waves theory to explain some particular issues arising in our approach. That is, the outcome of a rigid, impermeable boundary, together with a free surface where measurements are performed. We have generalized the works in \cite{alessandrini2000optimal,alessandrini2002detecting,beretta2017size}, by considering the Neumann and Dirichlet measurements simultaneously. Moreover if we allow changes of sign for the bottom difference, constants in the size estimate are different.

Concerning the general water--waves system and the present framework of the paper we have the following comments.

First, well-posedness of the general water--waves system in $\R^N$, on bounded domains, is an open question. Among others, because of the physical phenomena arising in the contact line between the free surface and the rigid solid walls \cite{graham1983new,kim2015capillary}.

Second, the water--waves system and asymptotic related systems are studied, classically, on unbounded domains (a strip). In \cite{fontelos2017bottom}, the authors proved the identifiability inverse problem of bottom detection by free surface measurements in that context. It would be interesting to state the results of this paper, as well as those in \cite{alessandrini2000optimal,alessandrini2002detecting,beretta2017size}, in the unbounded domain case, but some difficulties arise. The Lipschitz propagation of smallness and the stability estimates for ill-posed Cauchy problems are unknown in the strip-domain context.

Third, an interesting problem will be the study of the size estimate as in Case I, but from measurements on different free surfaces; that is, $\phi_0=\psi_0$ on $\Gamma(\zeta_0)$ and $\phi=\psi$ on $\Gamma(\zeta)$.

\section*{Acknowledgments:}
The authors thank M. Choulli for suggesting references \cite{choulli2019new, choulli2019global} and value comments on the stability of ill-posed Cauchy problems.

\bibliographystyle{abbrv}
\bibliography{bib}

\end{document}